\documentclass[11pt]{article}
\usepackage{amsmath,amssymb,amsbsy,amsfonts,amsthm,latexsym,
            amsopn,amstext,amsxtra,euscript,amscd,amsthm}

\newtheorem{lem}{Lemma}
\newtheorem{lemma}[lem]{Lemma}

\newtheorem{thm}{Theorem}
\newtheorem{theorem}[thm]{Theorem}

\def\\{\cr}
\def\({\left(}
\def\){\right)}
\def\[{\left[}
\def\]{\right]}
\def\<{\langle}
\def\>{\rangle}

\begin{document}

\title{On the Euler function of linearly recurrent sequences}

\date{\today}

\pagenumbering{arabic}

\author{Florian~Luca\footnote{Mathematics Division, Stellenbosch University, Stellenbosch, South Africa and Centro de Ciencias Matem\'aticas UNAM, Morelia, Mexico  email: florian.luca@wits.ac.za}\qquad Makoko Campbell Manape\footnote{School of Maths, Wits University, Johannesburg, South Africa}}

\maketitle

\begin{abstract}
In this paper, we show that if $(U_n)_{n\ge 1}$ is any nondegenerate linearly recurrent sequence of integers whose general term is up to sign not a polynomial in $n$, then the inequality $\phi(|U_n|)\ge |U_{\phi(n)}|$ holds on a set of positive integers $n$ of density $1$, where $\phi$ is the Euler function. In fact, we show that the set of $n\le x$ for which the above inequality fails has counting function $O_U(x/\log x)$.
\end{abstract}



\section{Introduction}

Let $(U_n)_{n\ge 1}$ be a linearly recurrent sequence of integers. Such a sequence satisfies a recurrence of the form
\begin{equation}
\label{eq:recurrence}
U_{n+k}=a_1U_{n+k-1}+\cdots+a_k U_n\qquad {\text{\rm for~all}}\qquad n\ge 1
\end{equation}
with integers $a_1,\ldots,a_k$, where $U_1,\ldots,U_{k}$ are integers. Assuming $k$ is minimal, $U_n$ can be represented as
\begin{equation}
\label{eq:Binet}
U_n=\sum_{i=1}^s P_i(n)\alpha_i^n,
\end{equation}
where 
\begin{equation}
\label{eq:Charpoly}
\Psi(X):=X^k-a_1X^{k-1}-\cdots-a_k=\prod_{i=1}^s (X-\alpha_i)^{\sigma_i}
\end{equation}
is the characteristic polynomial of $(U_n)_{n\ge 1}$, $\alpha_1,\ldots,\alpha_s$ are the distinct roots of $\Psi(X)$ with multiplicities $\sigma_1,\ldots,\sigma_s$, respectively, and $P_i(X)$ 
is a polynomial of degree $\sigma_i-1$ with coefficients in ${\mathbb Q}(\alpha_i)$. The sequence is nondegenerate if $\alpha_i/\alpha_j$ is not a root of $1$ for any $i\ne j$ in $\{1,\ldots,s\}$.
A classical example is the Fibonacci sequence $(F_n)_{n\ge 1}$ which has $k=2$, $\Psi(X)=X^2-X-1$ and initial terms $F_1=F_2=1$. Let $\phi(m)$ and $\sigma(m)$ 
be the Euler function and sum of divisors function of the positive integer $m$. In \cite{Luc}, the first author proved that the inequalities 
$$
\phi(F_n)\ge F_{\phi(n)}\qquad {\text{\rm and}}\qquad \sigma(F_n)\le F_{\sigma(n)}
$$
hold for all positive integers $n$. It was also remarked that if instead of considering $(F_n)_{n\ge 1}$ one considers a Lucas sequence with complex conjugated roots; i.e., a 
nondegenerate binary recurrent sequence $(U_n)_{n\ge 0}$ with $U_0=0,~U_1=1$ and $\Psi(X)$ a quadratic polynomial with complex conjugated roots, then the inequality 
$$
\phi(|U_n|)\ge |U_{\phi(n)}|
$$
fails infinitely often. In fact, it fails for a positive proportion of prime numbers $n$. Such questions were recently revisited by other authors (see \cite{Pra} and \cite{Pan}, for example). 

In this paper, we prove the following theorem. Recall that if $f(x)$ and $g(x)$ are functions defined on ${\mathbb R}_{+}$ with values in ${\mathbb R}_{+}$ we write $f(x)=O(g(x))$ and $f(x)=o(g(x))$ if the inequality $f(x)<Kg(x)$ holds with some constant $K>0$ and all $x>x_0$, and $\lim_{x\to\infty} f(x)/g(x)=0$, respectively. Further, the notations $f(x)\ll g(x)$ and $g(x)\gg f(x)$ are equivalent to $f(x)=O(g(x))$. When the implied constant $K$ depends on some other parameters like $U,~\varepsilon$, we indicate this by writing $f(x)=O_{U,\varepsilon}(g(x))$ or $f(x)\ll_{U,\varepsilon} g(x)$. 

\begin{theorem}
\label{thm:Main}
Let $(U_n)_{n\ge 1}$ be a nondegenerate linearly recurrent sequence of integers such that $|U_n|$ is not a polynomial in $n$ for all large $n$ and let $x$ be a large real number. Then the inequality 
\begin{equation}
\label{eq:1}
\phi(|U_n|)\ge |U_{\phi(n)}|
\end{equation}
fails on a set of positive integers $n\le x$ of cardinality $O_U(x/\log x)$. A similar statement holds for the positive integers $n\le x$ for which the inequality 
$$
\sigma(|U_n|)\le |U_{\sigma(n)}|
$$
fails. 
\end{theorem}

The theorem does not hold for sequences for which $U_n$ is either $P(n)$ or $(-1)^n P(n)$, with some polynomial $P(X)\in {\mathbb Z}[X]$, whose characteristic polynomial $\Psi(X)$ is one of $(X-1)^{k}$ or $(X+1)^{k}$, where $k-1$ is the degree of $P(X)$. 
For example, with $k=3$ and $P(X)=X^2+1$, we have that if $n$ is odd then $U_n=n^2+1$ is even therefore 
$$
\phi(U_n)\le \frac{n^2+1}{2}.
$$
On the other hand, for a positive proportion of $n$, we have $\phi(n)>n/{\sqrt{2}}$ and all such $n$ are odd. Indeed, if $n$ is even, then $\phi(n)/n\le 1/2$, so we cannot have $\phi(n)/n>1/{\sqrt{2}}$ for such $n$. 
To justify why there are a positive proportion of such $n$, recall that Schoenberg \cite{Scho} proved the existence of a continuous monotone function $f:[0,1]\rightarrow [0,1]$ with $f(0)=0$ and $f(1)=1$ such that 
$$
\lim_{x\to\infty} \frac{1}{x}\#\left\{n\le x:\frac{\phi(n)}{n}\le \alpha\right\}=f(\alpha)\qquad {\text{\rm for}}\qquad \alpha\in [0,1].
$$
In particular, the density of the set of $n$ such that $\phi(n)/n>1/{\sqrt{2}}$ equals $1-f(1/{\sqrt{2}})>0$. For such $n$, 
$$
U_{\phi(n)}=\phi(n)^2+1>\frac{n^2}{2}+1>\phi(U_n).
$$

As we said above, the bound $O_U(x/\log x)$ from the statement of Theorem \ref{thm:Main} is close to the truth in some cases like when $(U_n)_{n\ge 0}$ is a Lucas sequence with complex conjugated roots. Even more, it is easy to construct binary recurrent sequences $(U_n)_{n\ge 1}$ with real roots for which inequality \eqref{eq:1} fails for a number of positive integers $n\le x$ which is $\gg_U x/(\log x)$. For example, let 
$q_1<\cdots<q_k$ be odd primes such that 
$$
\sum_{i=1}^k \frac{1}{q_i}>1.
$$
Let $a>2$ be a positive integer such that $a\equiv 2\pmod {q_i}$ for $i=1,\ldots,k$. Then $2^p-a$ is a multiple of $q_i$ for all $i=1,\ldots,k$, whenever $p$ is a prime such that $p\equiv 1\pmod {q_i-1}$ for all $i=1,\ldots,k$. For such primes $p$ which are sufficiently large, we have 
\begin{eqnarray*}
\phi(2^p-a) & = & (2^p-a)\prod_{q\mid 2^p-a} \left(1-\frac{1}{q}\right)\le (2^p-a)\prod_{i=1}^k \left(1-\frac{1}{q_i}\right)\\
& < & (2^p-a) \exp\left(-\sum_{i=1}^k \frac{1}{q_i}\right)<\frac{2^p-a}{e}\\
& < & 2^{p-1}-a=2^{\phi(p)}-a.
\end{eqnarray*}
Thus, $\phi(U_n)<U_{\phi(n)}$ for $n=p$ a large prime in the progression 
$$p\equiv 1\pmod {{\text{\rm lcm}}[q_1-1,\ldots,q_k-1]}$$
and $U_n:=2^n-a$ which is the $n$th term of a binary recurrent sequence of characteristic polynomial $\Psi(X)=X^2-3X+2$. In  the above,  the notation ${\text{\rm lcm}}[q_1-1,\ldots,q_k-1]$ stands for the least common multiple of $q_1-1,\ldots,q_k-1$.

\section{Preliminary results}
\subsection{Arithmetic functions}

Here, we collect a few facts from the anatomy of integers which come in handy for our proof of Theorem \ref{thm:Main}. The first result addresses the minimal order of $\phi(n)$ and the maximal order of $\sigma(n)$. It follows from Theorems 323, 328 and 329 in \cite{HW}. 

\begin{lemma}
\label{lem:0}
Let $n\ge 3$. We then have 
$$
\frac{\phi(n)}{n}\gg \frac{1}{\log\log n}\qquad {\text{\rm and}}\qquad \frac{\sigma(n)}{n}\ll \log \log n.
$$
\end{lemma}

For a positive integer $n$ put $p(n)$ for the smallest prime factor of $n$ with the convention that $p(1)=1$. 
For $x\ge y\ge 2$ put
$$
\Phi(x,y):=\#\{n\le x: p(n)>y\}.
$$
The following inequality is a consequence of the Brun sieve and appears, for example, on page 397 in \cite{Ten} (see also Exercise on page 11 in \cite{HT}).
\begin{lemma}
\label{lem:1}
We have, uniformly for $x\ge y\ge 2$, 
$$
\Phi(x,y)\ll \frac{x}{\log y}.
$$
\end{lemma}

Let $\Omega(n)$ be the total number of prime factors of $n$ counting multiplicities.

\begin{lemma}
\label{lem:2} 
Let $x\ge 10$. The number of positive integers $n\le x$ such that $\Omega(n)\ge 10\log\log x$ is  $O(x/(\log x)^2)$. 
\end{lemma}

\begin{proof}
Exercise 05 on page 12 in \cite{HT} shows that 
$$
\#\{n\le x: \Omega(n)=k\}\ll \frac{xk\log x}{2^k}
$$
uniformly in $k\ge 1$ and $x\ge 2$. Taking $K:=\lfloor 10\log\log x\rfloor$ and applying the above estimate with $k\ge K$, we get that 
\begin{eqnarray*}
\#\{n\le x: \Omega(n)\ge 10\log\log x\} & \ll x & \log x\sum_{k\ge K} \frac{k}{2^k}\ll \frac{xK\log x}{2^K}\\
& \ll & \frac{x\log x\log\log x}{2^{10\log\log x}}=\frac{x\log x\log\log x}{(\log x)^{10\log 2}}\\
& = & O\left(\frac{x}{(\log x)^2}\right).
\end{eqnarray*}
\end{proof}
Let $\tau(n)$ be the number of divisors of $n$. 

\begin{lemma}
\label{lem:3}
Let $x\ge 10$. The number of positive integers $n\le x$ such that $\tau(\sigma(n))>\exp({\sqrt{\log x}})$ is $O(x/(\log x)^2)$. 
\end{lemma}

\begin{proof}
Theorem 1 in \cite{LP} shows that 
$$
\sum_{n\le x} \tau(\phi(n))=x\exp\left(c(x)\left(\frac{\log x}{\log\log x}\right)^{1/2} \left(1+O\left(\frac{\log\log\log x}{\log x}\right)\right)\right),
$$
where $c(x)\in [e^{-\gamma/2}/7,2{\sqrt{2}}e^{-\gamma/2}]$ and $\gamma=0.577\ldots$ is the Euler-Mascheroni constant. The remarks on page 128 of the same paper show that the above 
estimate holds with $\phi(n)$ replaced by $\sigma(n)$. In particular, 
\begin{eqnarray*}
\#\{n\le x: \tau(\sigma(n))>\exp({\sqrt{\log x}})\} \exp({\sqrt{\log x}})  &  \le &   \sum_{n\le x} \tau(\sigma(n))\\
& < & \exp\left(O\left(\frac{\log x}{\log\log x}\right)^{1/2}\right),
\end{eqnarray*}
which gives that 
\begin{eqnarray*}
\#\{n\le x: \tau(\sigma(n))>\exp({\sqrt{\log x}})\} & < & \exp\left(-{\sqrt{\log x}}+O\left(\left(\frac{\log x}{\log\log x}\right)^{1/2}\right)\right)\\
 & = & O\left(\frac{x}{(\log x)^2}\right).
 \end{eqnarray*}
\end{proof}

\subsection{ The subspace theorem and linearly recurrent sequences} 

Here, we review a quantitative version of the Subspace theorem due to Evertse from \cite{Ev} and apply it to nondegenerate linearly recurrent sequences of integers.
Let ${\mathbb K}$ be an algebraic number field with ring of integers ${\mathcal O}_{\mathbb K}$ and collection of places (equivalence classes of absolute values) $M_{\mathbb K}$. For $v\in M_{\mathbb K}$ and $x\in {\mathbb K}$, we define the absolute value $|x|_v$ as follows
$$
|x|_v:=\left\{\begin{matrix} 
|\sigma(x)|^{\frac{1}{[{\mathbb K}:{\mathbb Q}]}} & {\text{\rm if {\it v} corresponds to}} & \sigma:{\mathbb K}\mapsto {\mathbb R};\\
|\sigma(x)|^{\frac{2}{[{\mathbb K}:{\mathbb Q}]}} & {\text{\rm if {\it v} corresponds to}} & {\text{\rm the pair}}~~ \sigma,{\overline{\sigma}}: {\mathbb K}\mapsto {\mathbb C};\\
N(\pi)^{-\frac{{\text{\rm ord}}_{\pi} (x)}{[{\mathbb K}:{\mathbb Q}]}}& {\text{\rm if {\it v} corresponds to}} & {\text{\rm the prime ideal}}  ~~\pi\subset {\mathcal O}_{\mathbb K}.
\end{matrix}
\right.
$$
Here, $N(\pi):=\#({\mathcal O}_{\mathbb K}/\pi)$ is the norm of $\pi$ and ${\text{\rm ord}}_\pi(x)$ is the exponent of $\pi$ in the factorization of the principal fractional ideal $(x)$ of ${\mathbb K}$ with the convention that ${\text{\rm ord}}_{\pi}(0)=\infty$. In the first two cases above we call $v$ real infinite or complex infinite, respectively, while in the third case we call $v$ finite. These absolute values satisfy the product formula 
$$
\prod_{v\in M_{\mathbb K}} |x|_v=1\qquad {\text{\rm for~all}}\qquad x\in {\mathbb K}^*.
$$
Now let $s\ge 2$, ${\bf x}:=(x_1,\ldots,x_s)\in {\mathbb K}^s$ with ${\bf x}\ne 0$ and define 
$$
|{\bf x}|_v:=\left\{\begin{matrix} 
\left(\sum_{i=1}^s x_i^{2[{\mathbb K}:{\mathbb Q}]}\right)^{\frac{1}{2[{\mathbb K}:{\mathbb Q}]}} & {\text{\rm if {\it v} is real infinite}};\\
\left(\sum_{i=1}^s |x_i|^{[{\mathbb K}:{\mathbb Q}]}\right)^{\frac{1}{[{\mathbb K}:{\mathbb Q}]}} & {\text{\rm if {\it v} is complex infinite}};\\
\max\{|x_1|_v,\ldots,|x_s|_v\} & {\text{\rm if {\it v} is finite}}.
\end{matrix}
\right.
$$
Note that for infinite places $v$, $|\cdot|_v$ is a power of the Euclidean norm. Define
$$
H({\bf x}):=\prod_{v\in M_{\mathbb K}} |{\bf x}|_v.
$$
In the statement of the next result, the following apply:
\begin{itemize}
\item ${\mathbb K}$ is an algebraic number field;
\item ${\mathcal S}$ is a finite subset of $M_{\mathbb K}$ of cardinality $r$ containing all the infinite places;
\item $\{l_{1,v},\ldots,l_{s,v})$ for $v\in {\mathcal S}$ are linearly independent sets of linear forms with algebraic coefficients in $s$ variables such that 
$$
H(l_{i,v})\le H,\qquad [{\mathbb K}(l_{i,v)}:{\mathbb K}]\le D,
$$
for all $i=1,\ldots,s$ and $v\in {\mathcal S}$.
\end{itemize}

The following is the main Theorem in \cite{Ev}.

\begin{theorem}
\label{thm:ev}
Let $0<\delta<1$. Consider the inequality 
\begin{equation}
\label{eq:xxx}
\prod_{v\in {\mathcal S}} \prod_{i=1}^s \frac{|l_{i,v}({\bf x})|_v}{|{\bf x}|_v}<\left(\prod_{v\in {\mathcal S}} |{\text{\rm det}}(l_{1,v},\ldots,l_{s,v})|_v\right) H({\bf x})^{-s-\delta},\qquad {\bf x}\in {\mathbb K}^s.
\end{equation}
Then
\begin{itemize}
\item[(i)] There are proper linear subspaces $T_1,\ldots,T_{t_1}$ with 
$$
t_1\le (2^{60s^2}\delta^{-7s})^r\log(4D) \log\log(4D)
$$
such that every solution ${\bf x}\in {\mathbb K}^s$  to \eqref{eq:xxx} with $H({\bf x})\ge H$ satisfies 
$${\bf x}\in T_1\cup \cdots\cup T_{t_1}.$$
\item[(ii)] There are proper linear subspaces $S_1,S_2,\ldots,S_{t_2}$ of ${\mathbb K}^s$ with 
$$
t_2\le (150s^4\delta^{-1})^{sr+1} (2+\log\log (2H))
$$
such that every solution ${\bf x}\in {\mathbb K}^s$ of \eqref{eq:xxx} with $H({\bf x})< H$ satisfies 
$$
{\bf x}\in S_1\cup S_2\cup \cdots \cup S_{t_2}.
$$
\end{itemize}
\end{theorem}

We present an application to small values of nondegenerate linearly recurrent sequences. But before, let us record the following result of Schmidt \cite{Sch}. 
For a nondegenerate linearly recurrent sequence $(U_n)_{n\ge 1}$ let
$$
{\mathcal Z}_U:=\#\{n: U_n=0\}.
$$

\begin{theorem}
\label{thm:Schmidt}
If $(U_n)_{n\ge 1}$ is  a nondegenerate linearly recurrent sequence of order $k\ge 2$ whose terms are complex numbers, then 
$$
\#{\mathcal Z}_U\le \exp(\exp(\exp(3k\log k))).
$$
\end{theorem}

Now let $(U_n)_{n\ge 1}$ be a nondegenerate linearly recurrent sequence of integers given by recurrence \eqref{eq:recurrence}, whose characteristic polynomial is given by \eqref{eq:Charpoly} and 
formula for the general term \eqref{eq:Binet}. Assume that $|\alpha_1|\ge |\alpha_2|\ge \cdots \ge |\alpha_s|$ and that $|U_n|$ is not a polynomial in $n$ for large $n$. In particular, $|\alpha_1|>1$. 

We prove the following lemma.

\begin{lemma}
\label{lem:delta}
Let $(U_n)_{n\ge 1}$ be a nondegenerate  linearly recurrent sequence of integers whose general term is given by \eqref{eq:Binet} with $s\ge 2$ and assume that $|\alpha_1|=\max\{|\alpha_j|: 1\le j\le s\}$. Then there exists $x_0$ and $c:=c(U)\in (0,1/3)$ such 
that for $x\ge x_0$ the number of $n\le x$ such that 
\begin{equation}
\label{eq:smallUn}
|U_n|\le |\alpha_1|^{n(1-\delta)},
\end{equation}
with $\delta:=x^{-c}$ is of cardinality $O_U({\sqrt{x}})$.
\end{lemma}

\begin{proof}
We may assume that $n\in (x^{1/2},x]$ since there are only $O(x^{1/2})$ positive integers $n\le x^{1/2}$. Using \eqref{eq:Binet}, inequality \eqref{eq:smallUn} becomes 
$$
|\sum_{i=1}^s P_i(n)\alpha_i^n|\le |\alpha_1|^{n(1-\delta)}.
$$
Let $L$ be a positive integer which is a common multiple of all the denominators of all the coefficients of $P_{i}(X)$ for $i=1,\ldots,s$. Multiplying across by $L$, we get by setting $Q_i(X):=LP_i(X)$ that 
\begin{equation}
\label{eq:L}
|\sum_{i=1}^s Q_i(n)\alpha_i^n|\le L|\alpha_1|^{n(1-\delta)}.
\end{equation}
Note now that $Q_i(n)\alpha_i^n\in {\mathcal O}_{\mathbb K}$, where ${\mathbb K}:={\mathbb Q}(\alpha_1,\ldots,\alpha_s)$ is an algebraic number field. For technical reasons, we would like to get rid of the greatest common divisor of the ideals $(\alpha_1),\ldots,(\alpha_s)$. So, let $I:=\gcd((\alpha_1),\ldots,(\alpha_s))$. Then $I^h$ is principal for some positive integer $h$ which can be taken to be the cardinality of the class group of ${\mathbb K}$. Let $\beta$ be a generator of $I^h$. Then $\beta$ divides $\alpha_i^h$ for all $i=1,\ldots,s$, so $(\alpha_1^h/\beta),\ldots,(\alpha_s^h/\beta)$ 
are coprime. Since $I$ is Galois invariant, any conjugate $\beta^{(j)}$ of $\beta$ is also a generator of $I$, so $\beta$ is associated to any of its conjugates. Letting $d$ be the degree of $\beta$, we get that 
$\alpha_i^h/\beta^{(j)}$ are all associated for $j=1,\ldots,d$ (and fixed $i\in \{1,\ldots,s\}$) and in particular they are also associated to $\alpha_i^{hd}/b$, where we can take 
$b:=N(\beta)$. Now we replace $(U_n)_{n\ge 1}$ with any of the $hd$ linearly recurrent sequences $(U_{hdm+\ell})_{m\ge 0}$ and $\ell\in \{0,1,\ldots,hd-1\}$ by fixing $\ell$. Then 
$$
U_{hdm+\ell}=b^m\sum_{i=1}^s Q_i'(m)\alpha_i'^m,
$$
where $Q_i'(X):=\alpha_i^{\ell}Q_i(hdX+\ell)\in {\mathcal O}_{\mathbb K}[X]$ and $\alpha_i':=\alpha_i^{hd}/b$ for $i=1,\ldots,s$. Inequality \eqref{eq:L} now implies 
\begin{equation}
\label{eq:subs}
|\sum_{i=1}^s Q_i'(m)\alpha_i'^m|\le L|\alpha_1|^{\ell}\left|\frac{\alpha_1^{hd}}{b}\right|^m\cdot \alpha_1^{-\delta (hdm+\ell)}=L'|\alpha_1'|^{m(1-\delta_1)},
\end{equation}
where $L':=L|\alpha_1|^{\ell(1-\delta)}$ and $\delta_1:=c_0\delta$, with $c_0:={\displaystyle{\frac{hd\log |\alpha_1|}{hd\log |\alpha_1|-\log b}}}$. 
Note that $|\alpha_1'|> 1$, for if not, then $|\alpha_i'^{hd}/b|\le 1$ holds for all $i=1,\ldots,s$. 
Since $\alpha_i'$ are algebraic integers having all the conjugates at most $1$, we get that they are in fact roots of unity. Thus, $\alpha_i/\alpha_j$ is a root of unity for all $i\ne j$, which contradicts 
the nondegeneracy assumption. 

We now set up the subspace machinery. We let ${\mathcal S}$ be the subset of $M_{\mathbb K}$ containing all the infinite valuations as well as all the finite ones $v$ such that $|\alpha_i'|_v\ne 1$ 
for some $i=1,\ldots,s$. We take ${\bf x}=(x_1,\ldots,x_s)$ and $l_{i,v}({\bf x})$ given by 
$$
l_{i,v}({\bf x}):=x_i\qquad {\text{\rm for~all}}\qquad (i,v)\in \{1,\ldots,s\}\times {\mathcal S}~{\text{\rm with}}~i\ge 2~{\text{\rm or}}~v~{\text{\rm finite}},
$$
and take 
$$
l_{1,v}({\bf x}):=x_1+\cdots+x_s\qquad {\text{\rm for}}\qquad v~{\text{\rm infinite}}.
$$
We evaluate 
\begin{equation}
\label{eq:dp}
\prod_{i=1}^s \prod_{v\in {\mathcal S}} |l_{i,v}({\bf x})|_v
\end{equation}
in ${\bf x}:=(Q_1'(m)\alpha_1'^m,\ldots,Q_s'(m)\alpha_s'^m)$  with some $m$ satisfying inequality \eqref{eq:subs}. For a fixed $i\ge 2$, we have 
$$
\prod_{v\in {\mathcal S}} |l_{i,v}({\bf x})|_v=\prod_{v\in {\mathcal S}} |Q'_i(m)\alpha_i'^m|_v\le \prod_{\substack{v~{\text{\rm infinite}}}} |Q_i'(m)|_v\ll m^{\sigma_i-1},
$$
where the implied constant depends on the coefficients of $Q_i'(x)$. The above inequality follows by the product formula for $\alpha_i'^m$ together with the fact that ${\mathcal S}$ contains 
all the places of $M_{\mathbb K}$ for which $|\alpha_i'|_v\ne 1$, together with the fact that $Q_i'(m)$ is an algebraic integer so $|Q_i'(m)|_v\le 1$ for all finite places $v$. Hence, 
$$
\prod_{i=2}^s \prod_{v\in {\mathcal S}} |l_{i,v}({\bf x})|_v\ll \prod_{i=2}^s m^{\sigma_i-1}=m^{\sum_{i=2}^s \sigma_i-1}.
$$
For $i=1$, we have that 
\begin{eqnarray*}
\prod_{v\in {\mathcal S}} |l_{1,v}({\bf x})|_v & = & \prod_{\substack{v\in {\mathcal S}\\  v~{\text{\rm finite}}}} |Q_1(m)'\alpha_1'^m|_v\prod_{\substack{v\in {\mathcal S}\\ v~{\text{\rm infinite}}}} |\sum_{i=1}^s Q_i'(m)\alpha_i'^m|_v \\
& \ll & \prod_{\substack{v\in {\mathcal S}\\ v~{\text{\rm infinite}}}} |Q_i'(m)|_v \left(\prod_{\substack{v\in {\mathcal S}\\ v~{\text{\rm finite}}}} |\alpha_1'^m|_v \right) |\alpha_1'|^{m(1-\delta_1)}.
\end{eqnarray*}
In the above, we used that fact that $\sum_{i=1}^s Q_i'(m)\alpha_i'^m$ is an integer from ${\mathbb Z}$ so the product of its valuations over all infinite places $v\in M_{\mathbb K}$ is just the regular absolute value of this integer. Using again the product formula, $|\alpha_1'^m|$ is cancelled by the second product above, so we get that 
$$
\prod_{v\in {\mathcal S}} |l_{1,v}({\bf x})|_v \ll m^{\sigma_1-1} (|\alpha_1'^m|)^{-\delta_1}.
$$
Collecting everything together we get that the product shown in \eqref{eq:dp} is bounded as 
\begin{equation}
\label{eq:c1}
\prod_{i=1}^s \prod_{v\in {\mathcal S}} |l_{i,v}({\bf x})|_v\ll m^{\sum_{i=1}^s \sigma_i-1}( |\alpha_1'^m|)^{-\delta_1}\ll m^k |\alpha_1'^m|^{-\delta_1}.
\end{equation}
To be able to apply Theorem \ref{thm:ev}, we should compare the above upper bound on our double product with 
$$
\left(\prod_{v\in {\mathcal S}} |{\text{\rm det}}(l_{1,v},l_{2,v},\ldots,l_{s,v})|_v\right)\left(\frac{\prod_{v\in {\mathcal S}} |{\bf x}|_v}{H({\bf x})}\right)^s H({\bf x})^{-\delta_2},
$$
for some suitable $\delta_2$. Well, the first factor above is easy since all the involved determinants are equal to $1$. For the second factor above, we have that 
\begin{eqnarray*}
\frac{\prod_{v\in {\mathcal S}} |{\bf x|}_v}{H({\bf x})} & = & \prod_{v\in M_{\mathbb K}\backslash {\mathcal S}} |{\bf x}|_v^{-1}=\prod_{v\in {\mathbb K}\backslash {\mathcal S}} |Q_i'(m)\alpha_i'^m|_v^{-1}\\
& = & 
\prod_{v\in M_{\mathbb K}\backslash {\mathcal S}} (\max\{|Q_i'(m)|_v)^{-1}\ge 1. 
\end{eqnarray*}
In the above, we used the fact that $M_{\mathbb K}\backslash {\mathcal S}$ consists only of finite valuations $v$ for which $|\alpha_i'^m|_v=1$. 
Finally, for $H({\bf x})$ we use the fact that $x_i\in {\mathcal O}_{\mathbb K}$ to deduce that 
$$
H({\bf x})\le \prod_{v~{\text{\rm infinite}}} |{\bf x}|_v\le \left(\sum_{i=1}^s |Q_i(m)'^2\alpha_i'|^{2m}\right)^{1/2}\ll m^k|\alpha_1'|^m.
$$
Here, we used the fact that $\sum_{i=1}^s Q_i'(m)^2\alpha_i'^{2m}$ is an integer as the collection of numbers $\{Q_1'(m)^2\alpha_1^{2m},\ldots,Q_{s}'(m)\alpha_s^{2m}\}$ is Galois stable. In particular, we have that 
$$
\left(\prod_{v\in {\mathcal S}} |{\text{\rm det}}(l_{1,v},l_{2,v},\ldots,l_{s,v})|_v\right)\left(\frac{\prod_{v\in {\mathcal S}} |{\bf x}|_v}{H({\bf x})}\right)^s H({\bf x})^{-\delta_2}\gg m^{-k\delta_2} |\alpha_1'^m|^{-\delta_2}.
$$
So, the inequality \eqref{eq:xxx} will hold for us assuming that 
\begin{equation}
\label{eq:c2}
c_1m^k |\alpha_1'^m|^{-\delta_1}\le c_2 m^{-k\delta_2} |\alpha_1'^m|^{-\delta_2}
\end{equation}
holds, where $c_1$ and $c_2$ are the constants implied by the $\ll $ and $\gg $ symbols in \eqref{eq:c1} and \eqref{eq:c2}, respectively. We take $\delta_2:=\delta_1/2$ and the above inequalitybecomes equivalent to
$$
(c_1/c_2)m^{k(1+\delta_2)}<|\alpha_1'^m|^{\delta_1/2}.
$$
Taking logarithms, we get 
$$
k(1+\delta_2)\log m+\log(c_1/c_2)\le (\delta_1/2)(\log |\alpha_1'|) m.
$$
Since $n\le x$, then left--hand side is $O(\log x)$. Since $\delta_1=c_0\delta\gg x^{-1/3}$, $\log |\alpha_1'|>0$ and $m\gg n\gg {\sqrt{x}}$, it follows that the right--hand side above is $\gg x^{1/6}$. 
Thus, the last inequality above holds for $x>x_0$, where $x_0$ depends on $U$. We conclude that our ${\bf x}$ satisfies inequality \eqref{eq:xxx} with $\delta_2:=\delta_1/2$ and $x>x_0$.

We take a closer look at $H({\bf x})$. Since $(\alpha_1'),\ldots,(\alpha_s')$ are coprime, it follows that for every finite place $v\in M_{\mathbb K}$ there is $i\in \{1,\ldots,s\}$ such that $|\alpha_i'|_v=1$. This shows that for finite $v$, we have 
$$
|{\bf x}|_v\gg \min\{ |Q_i'(m)|_v\}\gg m^{-k}.
$$
Hence, 
\begin{equation}
\label{eq:c3}
H({\bf x})\gg m^{-rk} \prod_{v~{\text{\rm infinite}}} |Q_i(m)'\alpha_i'^m|_v \gg m^{-rk}|\alpha_1'|^m.
\end{equation}
Here, $r$ is the cardinality of ${\mathcal S}$. For our set-up the parameter $H$ can be taken to be ${\sqrt{s}}$. Since $m\gg n\gg x^{1/2}$, it follows that for large $x$ 
the inequality 
$$
c_3 m^{-rk} |\alpha_1'|^m\ge {\sqrt{s}},
$$
holds, where $c_3$ is the constant implied in \eqref{eq:c3}. Thus, for $x>x_0$, we have 
$$
H({\bf x})\ge H.
$$
Also, for us $D=1$ since $l_{i,v}({\bf x})$ have coefficients from ${\mathbb Z}$. So, by Theorem \ref{thm:ev}, there are proper subspaces $T_1,\ldots,T_{t_1}$ with 
$$
t_1\le (2^{60s^2} \delta_2^{-7s})^r \log 4 \log\log 4,
$$
such that ${\bf x}\in T_1\cup T_2\cup \cdots \cup T_{t_1}$. Each of the containments ${\bf x}\in T_j$ leads to an equation of the form
$$
\sum_{i=1}^s C_i^{(j)} Q_i'(m)\alpha_i'(m)=0,
$$
where ${\bf C}^{(j)}:=(C_1^{(j)},\ldots,C_s^{(j)})\in {\mathbb K}^s$ is not the zero vector. Each such equation signals that $m$ is in the set of zeros of a nondegenerate linearly recurrent sequence of order at most $k$ so there are at most $O_k(1)$ such values of $m$, where the constant in $O_k$ can be taken to be $\exp(\exp(\exp(3k\log k)))$  by Theorem \ref{thm:Schmidt}. So, it remains to understand the upper bound on $t_2$. But this is 
$$
(2^{60s^2+7s} c_0^{-7s} x^{7sc})^{r}\log 4\log\log 4.
$$
Taking $c:=1/(15sr)$, the above bound becomes 
$$
2^{(60s^2+7s)r} c_0^{-7s} x^{7/15} \log 4\log\log 4,
$$
and this is smaller than $x^{1/2}$ for $x>x_0$. This finishes the proof. 
\end{proof}

\section{The proof of Theorem \ref{thm:Main}}
\subsection{The case of the Euler $\phi$ function} 

Let $p(n)$ be the smallest prime factor of $n$. Let 
$$
{\mathcal A}_1(x)=\{n\le x: p(n)>x^{c_1}\}
$$
where $c_1\in (0,1/6)$ is a constant to be determined later. We show first that ${\mathcal A}_1(x)$ contains $O(x/\log x)$ positive integers $n\le x$ as $x\to\infty$. 
Indeed, putting $y:=x^{c_1}$, 
the set ${\mathcal A}_1(x)$ coincides with the $n\le x$ which are coprime to all the primes $p\le y$. The number of such is, by Lemma \ref{lem:1},  
$$
\Phi(x,y) \ll \frac{x}{\log y}\ll \frac{x}{\log x}.
$$
From now on, let $n\le x$ not in ${\mathcal A}_1(x)$. We also assume that $n\ge x^{	1/2}$, since there are only $O(x^{1/2})=o(x/\log x)$ as $x\to\infty$ positive integers failing this last inequality. 

For such $n$, the interval $[1,n]$ contains at least $n/p(n)$ numbers which are not coprime to $n$, namely all the positive integers which are multiples of $p(n)$. Thus, 
$$
\phi(n)\le n-\frac{n}{p(n)}\le n-n\delta,
$$
where $\delta:=1/x^{c_1}$. 
Let
$$
U_n=\sum_{i=1}^s P_i(n)\alpha_i^n.
$$
We assume that $|\alpha_1|\ge |\alpha_2|\ge \cdots \ge |\alpha_s|$. Assume first that $s=1$. In this case $\Psi(X)=(X-\alpha_1)^k$, so $\alpha_1$ is an integer with $|\alpha_1|\ge 2$. Thus, 
$$
U_n=P_1(n)\alpha_1^n,
$$
where $P_1(X)\in {\mathbb Q}[X]$. Let $L$ be the least common denominator of all the coefficients of $P_1(X)$. Then, for large $n$ (say larger than the maximal real root of $P_1(X)$), we have 
\begin{eqnarray*}
L\phi(|U_n|) & \ge & \phi(L|U_n|)\ge \phi(L|P_1(n)|) \phi(|\alpha_1|^n|)\gg \frac{L|P_1(n)|}{\log\log(L|P_1(n)|)} |\alpha_1|^n\\
& \gg & \frac{Ln^{k-1}}{\log\log n} |\alpha_1|^n.
\end{eqnarray*}
This gives 
\begin{equation}
\label{s=11}
\phi(|U_n|)\gg \frac{n^{k-1}}{\log\log n} |\alpha_1|^n.
\end{equation}
On the other hand 
\begin{equation}
\label{s=12}
|U_{\phi(n)}|=|P_1(\phi(n))||\alpha_1|^{\phi(n)}\ll \phi(n)^{k-1} |\alpha_1|^{n(1-\delta)}\ll n^{k-1} |\alpha_1|^{n(1-\delta)}.
\end{equation}
By \eqref{s=11} and \eqref{s=12}, it follows that if 
$$
\phi(|U_n|)\le |U_{\phi(n)}|
$$
holds, then 
$$
\frac{n^{k-1}}{\log\log n} |\alpha_1|^n\le \phi(|U_n|)\le |U_{\phi(n)}|\ll n^{k-1}|\alpha_1|^{n(1-\delta)},
$$
This is equivalent to 
$$
|\alpha_1|^{n\delta}\ll \log\log n.
$$
Taking logarithms this becomes 
$$
(n\delta) \log |\alpha_1|\le \log\log\log n+O(1).
$$
The right--hand side is $O(\log\log\log x)$. Since $|\alpha_1|\ge 2$, $n\delta\ge x^{1/2}/x^{c_1}\ge x^{1/3}$, it follows that the above inequality implies that $x$ is bounded. Thus, there are only finitely many such $n$ 
in case $s=1$.

From now on, we assume $s\ge 2$. In this case, the inequality
$$
|\alpha_1|^{n/2}\le |U_n|\ll |\alpha_1|^n n^k
$$
holds for all $n\ge n_0$. Indeed, the right--hand side is obvious and the left--hand side follows from an application of the subspace theorem similar to Lemma \ref{lem:delta} 
except that now we can take $\delta:=1/2$  fixed (not tending to zero with $x$) and we get only finitely many positive integers $n$ for which the left--hand side inequality above does not hold.  Thus, for $n\ge n_0$, we have that 
$$
\log\log |U_n|=\log n+O(1).
$$
Since $\phi(m)\gg m/\log\log m$ holds for all integers $m\ge 2$ (see Lemma \ref{lem:0}), we have that for $n>n_0$, the inequality
$$
\phi(|U_n|)\gg \frac{|U_n|}{\log\log |U_n|}=\frac{|U_n|}{\log n+O(1)}
$$
holds. Assume now that 
$$
\phi(|U_n|)\le |U_{\phi(n)}|.
$$
We then get 
$$
\frac{|U_n|}{\log n+O(1)}\ll \phi(|U_n|)\le |U_{\phi(n)}|\le |\alpha_1|^{\phi(n)} \phi(n)^k\ll |\alpha_1|^{n-n\delta} n^k.
$$
This gives 
$$
|U_n|\ll |\alpha_1|^{n-n\delta} n^k(\log n+O(1)).
$$
Let $c_4$ be the constant implied by the above inequality and $c_5$ be the constant implied by the above $O(1)$. We claim that with $\delta_1:=1/x^{2c_1}$, the inequality 
$$
c_4 |\alpha_1|^{n(1-\delta)} n^k (\log n+c_5)<|\alpha_1|^{n(1-\delta_1)}
$$
holds. Indeed, this is equivalent to 
\begin{equation}
\label{eq:t}
k\log n+\log(\log n+c_5)+\log(c_4)<n(\delta-\delta_1)\log |\alpha_1|.
\end{equation}
Since $n\in ({\sqrt{x}},x]$, the left--hand side above is $O(\log x)$. Since $\delta_1=1/x^{2c_1}$ and $\delta=1/x^{c_1}$, it follows that $\delta-\delta_1\ge 0.5\delta\ge 0.5x^{-c_1}$ for $x\ge x_0$. Since $n\ge {\sqrt{x}}$, it follows that the right--hand side above is $\gg x^{1/2-c_1}\gg x^{1/3}$ therefore indeed \eqref{eq:t} holds for all our $n$ in $({\sqrt{x}},x]$ and $x>x_0$.
Thus, we get
$$
|U_n|\le |\alpha_1|^{n(1-\delta_1)}.
$$
By Lemma \ref{lem:delta}, we can choose $c_1:=c/2$, such that the number of $n\le x$ satisfying the above inequality is $O_k({\sqrt{x}})=o(x/\log x)$ as $x\to\infty$, which finishes the argument. 

\subsection{The case of the $\sigma$ function}

Assume again that $x$ is large and $n\in ({\sqrt{x}}, x]$ is divisible by some prime $p\le x^{c_1}$, for some small constant $c_1$, since otherwise, like in the case of the Euler function, the set of such $n\le x$
is $O(x/\log x)$. Then $\sigma(n)\ge n+n\delta$, where $\delta:=1/x^{c_1}$, which gives 
$$
n\le \frac{\sigma(n)}{1+\delta}\le \sigma(n)(1-\delta_1),
$$
where $\delta_1:=\delta/2$. Assume now that 
\begin{equation}
\label{eq:sigmasigma}
|U_{\sigma(n)}|\le \sigma(|U_n|).
\end{equation}
As in the case of the $\phi$ function, we need to treat the case $s=1$ separately. In this case, $U_n=P_1(n)\alpha_1^n$, where $|\alpha_1|\ge 2$ is an integer and $P_1(X)\in {\mathbb Q}[X]$. Let again $L$ be the least common denominator of the coefficients of $P_1(X)$ and $n$ be larger than the maximal real zero of $P_1(X)$. The right--hand side above is, 
by Lemma \ref{lem:0}.
\begin{eqnarray}
\label{eq:s=11sig}
\sigma(|U_n|)\le \sigma(L|U_n|) & = & \sigma(L|P_1(n)|)|\alpha_1|^n)\le \sigma(L|P_1(n)|)\sigma(|\alpha_1|^n)\nonumber\\
& \ll  &L|P_1(n)|(\log\log(L|P_1(n)|)) |\alpha_1|^n\nonumber\\
& \ll & n^{k-1}(\log\log n) |\alpha_1|^{\sigma(n)(1-\delta_1)},
\end{eqnarray}
while
\begin{equation}
\label{eq:s=12sig}
|U_{\sigma(n)}|=|P_1(\sigma(n))|\alpha_1|^{\sigma(n)}\gg \sigma(n)^{k-1} |\alpha_1|^{\sigma(n)}\gg n^{k-1} |\alpha_1|^{\sigma(n)}.
\end{equation}
Inequality \eqref{eq:sigmasigma} together with \eqref{eq:s=11sig} and \eqref{eq:s=12sig} imply 
$$
n^{k-1}|\alpha_1|^{\sigma(n)}\ll |U_{\sigma(n)}|\le \sigma(|U_n|)\ll n^{k-1}(\log\log n) |\alpha_1|^{\sigma(n)(1-\delta_1)}.
$$
This leads to 
$$
|\alpha_1|^{\delta_1\sigma(n)}\ll \log\log n,
$$
and by the argument for the case $s=1$ and the $\phi$ function, this leads to the conclusion that $x$ (so, $n$) is bounded.

From now on, we assume that $s\ge 2$. The right--hand side is, by Lemma \ref{lem:0} and the calculation done at the case of the Euler $\phi$ function,
\begin{eqnarray*}
\sigma(|U_n|) & \ll  & |U_n|\log\log |U_n|\ll |\alpha_1|^n n^k (\log n+O(1))\\
& \le & |\alpha_1|^{\sigma(n)(1-\delta_1)} \sigma(n)^k (\log \sigma(n)+O(1)).
\end{eqnarray*}
Let $c_6$ and $c_7$ be the constants implied by the $\ll$--symbol and $O$--symbol above, respectively. By the argument done in the case of the Euler $\phi$ function, putting $\delta_2:=1/x^{2c_1}$, the inequality 
$$
 c_6|\alpha_1|^{m(1-\delta_1)} m^k (\log m+c_7))<|\alpha|^{m(1-\delta_2)}
 $$
 holds for all $m=\sigma(n)$ and $n\in ({\sqrt{x}},x]$ for $x>x_0$. Thus, putting $m:=\sigma(n)$, we get that 
 $$
 |U_m|\le |\alpha_1|^{m(1-\delta_2)}
 $$
 holds when $x\ge x_0$. Note that $m\ll x\log\log x$ by Lemma \ref{lem:0}.  By Lemma \ref{lem:delta}, we can choose $c_1=c/2$ and then the set of $m\ll x\log\log x$ satisfying the above inequality is of cardinality
 $$
 O_k({\sqrt{x\log\log x}}).
 $$
But this is only an upper bound on the number of distinct values of $\sigma(n)$ and we have to get an upper bound on the number of $n$'s themselves. Well, by Lemmas \ref{lem:2} and \ref{lem:3}, we may assume that $\Omega(n)\le 10\log\log x$ and $\tau(\sigma(n))\le \exp({\sqrt{\log x}})$ since the number of $n\le x$ for which one of the above inequalities fails if $O(x/(\log x)^2)$. Well, let us write 
$$
n=p_1^{a_1}\cdots p_\ell^{a_{\ell}},
$$
with distinct primes $p_1,\ldots,p_{\ell}$ and positive exponents $a_1,\ldots,a_{\ell}$, so 
$$
\sigma(n)=\prod_{i=1}^{\ell} \left(\frac{p_i^{a_i+1}-1}{p_i-1}\right).
$$
Given $m=\sigma(n)$, each of $(p_i^{a_i+1}-1)/(p_i-1)$ is a divisor $d_i$ of $\sigma(n)$. Additionally, given $d_i$ and also $a_i$, $p_i$ is uniquely determined. Thus, since $d_i$ can be fixed in at most $\tau(\sigma(n))$ ways and $a_i\le \Omega(n)$ can be fixed in at most $\Omega(n)$ ways, it follows that $p_i^{a_i}$ can be fixed in at most $\Omega(n)\tau(\sigma(n))$ ways. This is so for a fixed $i$, but $i\le \ell=\omega(n)\le \Omega(n)$. Thus, the number of such $n$ when both $\sigma(n)$ and $\Omega(n)$ are given is at most 
$$
\left((10\log \log x )\exp({\sqrt{\log x}})\right)^{10\log\log x}<\exp\left(20 (\log\log x){\sqrt{\log x}}\right)
$$
for $x>x_0$. Varying $\Omega(n)$ up to $10\log\log x$ as well as the number of possible values of $\sigma(n)$, we get that the number of possible $n\le x$ is 
$$
\ll_k  {\sqrt{x\log\log x}}(\log\log x)\exp\left(20 (\log\log x){\sqrt{\log x}}\right)=o(x/\log x)
$$
as $x\to\infty$, which finishes the proof of the $\sigma$ case.

\section*{Acknowledgements}

We thank an annonymous referee for constructive criticism which improved the quality of our manuscript.


\begin{thebibliography}{9999}
 
 \bibitem{Ev} J.--H. Evertse, ``An improvement of the quantitative subspace theorem", {\it Compositio Math.\/} {\bf 101} (1996),  225--311. 
 
 \bibitem{HT} R. Hall and G. Tenenbaum, {\it Divisors\/}, Cambridge Tracts in Mathematics {\bf 90}, Cambridge University Press, Cambridge, 1988.
 
 \bibitem{HW} G. H. Hardy and E. M.  Wright, {\it An introduction to the theory of numbers\/}. Sixth edition. Revised by D. R. Heath-Brown and J. H. Silverman. With a foreword by Andrew Wiles. Oxford University Press, Oxford, 2008.
 
 \bibitem{Pra} M.  Jaidee and P. Pongsriiam, ``Arithmetic functions of Fibonacci and Lucas numbers", {\it Fibonacci Quart.\/} {\bf 57} (2019), 246--254.
 
 \bibitem{Luc} F. Luca, ``Arithmetic functions of Fibonacci numbers", 
{\it Fibonacci Quart.\/} {\bf 37} (1999), 265--268. 

\bibitem{LP}  F. Luca and C. Pomerance,  ``On the average number of divisors of the Euler function", {\it Publ. Math. Debrecen \/} {\bf 70} (2007), 125--148.
 
 \bibitem{Pan} A. K. Pandey and B. K. Sharma, ``On inequalities related to a generalized Euler totient function and Lucas sequences", {\it J. Integer Seq.\/} {\bf 26} (2023), Art. 23.8.6.
 
 \bibitem{Sch} W. M. Schmidt, ``The zero multiplicity of linear recurrence sequences", {\it  Acta Math.\/} {\bf 182} (1999),  243--282. 
 
 \bibitem{Scho} I. J. Schoenberg, ``\"Uber die asymptotische Verteilung reeler Zahlen mod 1", {\it Math. Z.\/} {\bf  28} (1928) 171--199.
 
 \bibitem{Ten} G. Tenenbaum, {\it Introduction to Analytic and Probabilistic Number Theory\/}, Cambridge Studies in Advanced Mathematics {\bf 46}, Cambridge U. Press, 1995.
 
 \end{thebibliography}
\end{document}